\documentclass[12pt]{amsart}
\usepackage{amsfonts}
\usepackage{amsmath}
\usepackage{amssymb}
\usepackage[margin=1.2in]{geometry}
\usepackage{verbatim}
\newtheorem{theorem}{Theorem}[section]

\newtheorem{lemma}[theorem]{Lemma}
\newtheorem{proposition}[theorem]{Proposition}

\theoremstyle{definition}

\newtheorem{remark}[theorem]{Remark}

\newcommand{\ec}{\exp^{\ast}}

\numberwithin{equation}{section}

\begin{document}
\title[On the estimate $M(x)=o(x)$]{On the estimate $M(x)=o(x)$ for Beurling generalized numbers}

\author[J.~Vindas]{Jasson Vindas} 
\thanks{This work was partially supported by the Research
  Foundation--Flanders, through the FWO-grants 1510119N and G067621N} 
\address{Department of Mathematics: Analysis, Logic and Discrete Mathematics\\ Ghent University\\
  Krijgslaan 281\\ B 9000 Ghent\\ Belgium} 
\email{jasson.vindas@UGent.be}


\begin{abstract} We show that the sum function of the M\"{o}bius function of a  Beurling number system must satisfy the asymptotic bound $M(x)=o(x)$ if it satisfies the prime number theorem and its prime distribution function arises from a monotone perturbation of either the classical prime numbers or the logarithmic integral.

\end{abstract}

\maketitle

\begin{center}

\emph{Dedicated to the memory of Jean-Pierre Kahane and Wen-Bin Zhang}

\end{center}

\bigskip

\section{Introduction}  \label{Sec Intro}
In analytic number theory, the Prime Number Theorem (PNT)
$$
\pi(x)\sim \frac{x}{\log x}, \qquad x\to \infty,
$$
is considered to be equivalent to the bound 
\begin{equation}
\label{Meeq1}
M(x)=o(x), \qquad x\to \infty,
\end{equation} 
where $M$ stands for the sum function of the M\"{o}bius function, in the sense that they are deducible from one another by quite simple real variable arguments (cf. \cite{diamond82,landau}). The counterparts of this and several other classical ``PNT equivalences'' for Beurling generalized numbers  have been thoroughly investigated  in the past decade \cite{d-d-v,d-m-v,d-vPNTequiv2017, d-z2012, diamond-zhangbook}; see also the survey articles \cite{diamond2023,d-z2017}. It has been discovered that various of the different implications may fail without extra hypotheses, and conditions have also been determined under which the implications do or do not hold in the Beurling context. 

Based on an example of Beurling \cite{beurling1937}, Zhang observed \cite{zhang87} that the Beurling analog of \eqref{Meeq1} does not necessarily imply the PNT for Beurling primes. In the converse direction, knowing whether \eqref{Meeq1} is an unconditional consequence of the PNT is an important \emph{open question} in the theory of Beurling numbers. 

According to \cite[Theorem 2.1]{d-d-v}, the PNT does imply \eqref{Meeq1} if one additionally assumes that the generalized integer counting function $N$ satisfies the bound $N(x)\ll x$. Therefore, if one would like to attempt to disprove the implication ``PNT $\Rightarrow$ \eqref{Meeq1}'', one would need to seek counterexamples among those number systems for which $N(x)/x$ is unbounded. This observation in combination with \cite[Proposition~6.1, p.~60]{diamond-zhangbook} led Kahane to propose\footnote{Orally communicated in April 2017 during an author's visit to Jean-Pierre Kahane and Eric Sa\"{i}as in Paris.} the absolutely continuous prime measure\footnote{Even if one would mainly be interested in discrete Beurling number systems, considering ``continuous analogs'' is a useful device for the construction of counterexamples in Beurling number theory, as they provide flexible templates that could usually be discretized by means of a number of available techniques (cf. \cite{b-v2024,d-s-v,diamond-zhangbook}). So, in a broad sense \cite{beurling1937,diamond-zhangbook}, a Beurling generalized number system is merely a pair of non-decreasing right continuous functions $N$ and $\Pi$ with $N(1)=1$ and $\Pi(1)=0$, both having support in $[1,\infty)$, and subject to the relation $\mathrm{d}N=\exp^{\ast}(\mathrm{d}\Pi)$.} (we write $\log_{k}x= \log (\log_{k-1} x)$)
\begin{equation}
\label{Meeq2}
\mathrm{d}\Pi_{\text{K}}(u)= \frac{1-u^{-1}}{\log u} \mathrm{d}u+ \frac{\chi_{[e^{e}, \infty)}(u)}{\log u \log_{2} u}  \mathrm{d}u
\end{equation}
as a possible candidate to refute ``PNT $\Rightarrow$ \eqref{Meeq1}''. In fact, we obviously have $\Pi_{\text{K}}(x)\sim x/\log x $, while the quoted result from \cite{diamond-zhangbook} yields $N_{\text{K}}(x)/x\to\infty$ for its associated generalized integer distribution function, determined via the relation $\mathrm{d}N_{K}=\ec(\mathrm{d}\Pi_{\text{K}})$ where the exponential is taken with respect to  the (multiplicative) convolution of measures \cite{diamond-zhangbook}. The analog of the sum of the M\"{o}bius function in this case is $M_{K}$ with $\mathrm{d}M_{K}$ defined as the convolution inverse of $\mathrm{d}N_{\text{K}}$, or equivalently, $\mathrm{d}M_{\text{K}}=\ec(-\mathrm{d}\Pi_{\text{K}})$.

We have found that $M_{K}(x)=o(x)$; consequently, the prime measure \eqref{Meeq2} fails to deliver a counterexample for ``PNT $\Rightarrow$ \eqref{Meeq1}''. This follows at once from the following theorem, which is our main result. We mention that primitives of all  measures considered in this article are right-continuous and supported on the interval $[1,\infty)$.

\begin{theorem}\label{Meth1}
Suppose the (signed) measure $\mathrm{d}\Pi$ can be written as
$$
\mathrm{d}\Pi= \mathrm{d}\Pi_{0}+ \mathrm{d}E+ \mathrm{d}R ,
$$
where the measures in this expression satisfy:
\begin{equation}\label{Meeq3} M_0(x)=\int_{1^{-}}^{x}\ec(-\mathrm{d}\Pi_{0})= O\left(\frac{x}{\log^{a}x}\right), \qquad \mbox{for some } a>0,
\end{equation}
\begin{equation}\label{Meeq4} \int_{1^{-}}^{x}|\mathrm{d}E(u)|= o\left(\frac{x}{\log x}\right),
\end{equation}
and
\begin{equation}\label{Meeq5} \int_{1^{-}}^{\infty}\frac{|\mathrm{d}R(u)|}{u}< \infty.
\end{equation}
Then, $M(x)= \int_{1^{-}}^{x}\ec(-\mathrm{d}\Pi)=o(x)$.

\end{theorem}

Theorem \ref{Meth1} is applicable when $\Pi_0$ is Riemann's counting function of the classical rational prime numbers $\Pi_{0}(x)=\sum_{k=1}^{\infty} (1/k)\pi(x^{1/k})= \sum_{p^{k}\leq x} (1/k)$, or when it is the (normalized) logarithmic integral $\Pi_0(x)= \int_{1}^{x} (1-u^{-1}) \mathrm{d}u/\log u$. In particular, any example of a prime measure constructed as a monotone perturbation of either of these prime distribution functions and for which the PNT holds will always fail to disprove ``PNT$\Rightarrow$\eqref{Meeq1}''.

We give a proof of Theorem \ref{Meth1} in Section \ref{Sec 3 Me}, after discussing some preliminaries on convolution of measures in Section \ref{Preli Me}. We shall close this article by exhibiting a better bound than $M_{K}(x)=o(x)$ for Kahane's example. The next proposition will be shown in Section \ref{Sec 4 Me}.
\begin{proposition}
\label{Mep1} We have the bound
\[M_{K}(x)=o\left(\frac{x }{\log x }\right).\]
\end{proposition}

The author thanks Frederik Broucke, Gregory Debruyne, and Harold G. Diamond  for their useful comments.

\section{Preliminaries on convolution calculus}\label{Preli Me}
 Our notation for operational calculus with Lebesgue-Stieltjes measures is the same as in \cite{diamond-zhangbook}, but let us briefly review some concepts for the sake of convenience. Integration $\int_{1^{-}}^{x}$ has the meaning $\int_{[1,x]}$. The convolution of $\mathrm{d}A$ and $\mathrm{d}B$ is determined via
$$
\int_{1^{-}} ^{x} \mathrm{d}A \ast \mathrm{d}B= \iint_{1\leq uv\leq  x} \mathrm{d}A (u) \mathrm{d}B(v)= \int_{1^{-}} ^{x} A\left(\frac{x}{u}\right)\mathrm{d}B(u)=\int_{1^{-}} ^{x} B\left(\frac{x}{u}\right)\mathrm{d}A(u).
$$
The variation measure of $\mathrm{d}A$ is denoted as $|\mathrm{d}A|$, a notation that was already used in the Introduction.

We shall make use of the operator $L$ of multiplication by $\log$, namely,
$$
(L\mathrm{d}A)(u)= \log u \: \mathrm{d}A(u).
$$
The relevance of $L$ for us is that it is a derivation in the convolution algebra of measures, that is, 
\begin{equation}
\label{Meeq2.1}
L(\mathrm{d}A\ast \mathrm{d}B)= (L\mathrm{d}A)\ast \mathrm{d}B+ \mathrm{d}A\ast (L \mathrm{d}B).
\end{equation}

Applying $L$ to the exponential of a measure $\ec(\mathrm{d}A)=\sum_{n=0}^{\infty} (1/n!) \mathrm{d}A^{\ast n}$ yields \cite[p.~27]{diamond-zhangbook}
\begin{equation}
\label{Meeq2.2}
L\ec(\mathrm{d}A)= (L \mathrm{d}A)\ast \ec(\mathrm{d}A) .
\end{equation}
The reader should think of \eqref{Meeq2.2} as a general form of Chebyshev's classical identity from elementary prime number theory\footnote{That is, the identity $\sum_{n\leq x}\log n=\sum_{n\leq x}\psi(x/n)$, where $\psi$ is the Chebyshev function.}. 

We finally recall the exponential law for convolution: 
\[\ec(\mathrm{d}A+\mathrm{d}B)=\ec(\mathrm{d}A)\ast \ec(\mathrm{d}B).\]
\section{Proof of Theorem \ref{Meth1}}\label{Sec 3 Me}

We divide the proof of Theorem \ref{Meth1} into several intermediate results.
\begin{lemma}\label{Mel1} Suppose $\mathrm{d}G_{1}$ and $\mathrm{d}G_2$ are such that 
\begin{equation}
\label{Meeq3.0}
\int_{1^{-}}^{\infty}\frac{|\mathrm{d}(G_1-G_2)(u)|}{u}<\infty
\end{equation}
and let $\mathrm{d}M_{j}=\ec(\mathrm{d}G_j)$, $j=1,2$. Then, 
$$
M_1(x)=o(x) \quad \mbox{if and only if} \quad M_2(x)=o(x). $$
\end{lemma}
\begin{proof} Since the statement is symmetric, we just need to show that $M_1(x)=o(x)$ implies $M_2(x)=o(x)$. Set $\mathrm{d}R=\mathrm{d}G_2-\mathrm{d}G_1$. Our hypothesis is that \eqref{Meeq5} holds, which in turn yields
$$
\int_{1^{-}}^{\infty} \frac{|\ec(\mathrm{d}R(u))|}{u}\leq  \int_{1^{-}}^{\infty} \frac{\ec(|\mathrm{d}R(u)|)}{u}=\exp \left(\int_{1^{-}}^{\infty} \frac{|\mathrm{d}R(u)|}{u}\right)<\infty.
$$
Hence,
\begin{align*}
M_2(x)&=\int_{1^{-}}^{x} \mathrm{d}M_1\ast \ec(\mathrm{d}R)=\int_{1^{-}}^{x} M_1\left(\frac{x}{u}\right)\ec(\mathrm{d}R(u))
\\
&= \int_{1^{-}}^{x} o\left(\frac{x}{u}\right)\ec(\mathrm{d}R(u))= o(x).
\end{align*}\end{proof}

In view of Lemma \ref{Mel1}, we may assume that $R=0$ in Theorem \ref{Meth1}. Using the Jordan decomposition of $\mathrm{d}E$, Theorem \ref{Meth1} would be established if we are able to show the next proposition.

\begin{proposition}
\label{Mep2} Let $\mathrm{d}E\geq0$ and $\mathrm{d}\Pi_0$  be measures such that \eqref{Meeq3} and \eqref{Meeq4} hold. Then, for each $0<b<\min\{a,1\}$,
\begin{equation}
\label{Meeq3.1} M^{\pm}(x)=\int_{1^{-}}^{x} \ec(-\mathrm{d}\Pi_0\pm \mathrm{d}E)=  O\left(\frac{x}{\log^{b}x}\right).
\end{equation}
\end{proposition}

The remainder of this section is devoted to proving Proposition \ref{Mep2} and we keep $\mathrm{d}\Pi_0$, $\mathrm{d}E$, $\mathrm{d}M_0$, and $\mathrm{d}M^{\pm}$ as in its statement. Naturally, \eqref{Meeq4} amounts to the same as 
\begin{equation}
\label{Meeq3.2}E(x)=o\left(\frac{x}{\log x}\right).
\end{equation}
 We further define
$$
\mathrm{d}F^{\pm}=\ec(\pm \mathrm{d}E) \qquad \mbox{and} \qquad \mathrm{d}H^{\pm}= L \mathrm{d}F^{\pm}= \log \cdot\, \mathrm{d}F^{\pm}.
$$

We need  some bounds involving $\mathrm{d}F^{\pm}$ and $H^{\pm}$.
\begin{lemma}
\label{Mel2} For each $\varepsilon>0$,
$$
\int_{1^{-}}^{x} \frac{\mathrm{d}F^{+}(u)}{u}\ll \log^{\varepsilon}x  \qquad \mbox{and} \qquad H^{+}(x)\ll x\log^{\varepsilon}x
.
$$
\end{lemma}
\begin{proof} By \eqref{Meeq3.2},
$$
\int_{1^{-}}^{\infty} \frac{\mathrm{d}E(u)}{u^{\sigma}}= \sigma \int_{1}^{\infty} \frac{E(u)}{u^{\sigma+1}} \mathrm{d}u= o\left(-\log(\sigma-1) \right), \qquad \sigma\to 1^{+}.
$$
Hence,
\begin{equation}
\label{eqextra}
\int_{1^{-}}^{\infty} \frac{\mathrm{d}F^{+}(u)}{u^{\sigma}}=\exp\left( \int_{1^{-}}^{\infty} \frac{\mathrm{d}E(u)}{u^{\sigma}}\right)\ll \frac{1}{(\sigma-1)^{\varepsilon}}\: .
\end{equation}
An quick elementary Tauberian argument shows that 
$
\int_{1^{-}}^{x} u^{-1}\mathrm{d}F^{+}(u)\ll \log^{\varepsilon}x.
$
Indeed, taking $\sigma=1+1/\log x$ in \eqref{eqextra}, we obtain
\[\int_{1^{-}}^{x} \frac{\mathrm{d}F^{+}(u)}{u}\leq e \int_{1^{-}}^{\infty}\exp\left(-\frac{\log u}{\log x}\right)\frac{\mathrm{d}F^{+}(u)}{u}\ll \log^{\varepsilon} x.
\]

Noticing now that $\tilde{E}(x)=\int_{1^{-}}^{x} \log u \, \mathrm{d}E(u)= E(x) \log x-\int_{1}^{x} u^{-1}E(u)\, \mathrm{d}u=o(x)$ and employing property \eqref{Meeq2.2}, we deduce that
$$
H^{+}(x)= \int_{1^{-}}^{x} \mathrm{d}\tilde{E} \ast \mathrm{d}F^{+}= \int_{1^{-}}^{x} \tilde{E}\left(\frac{x}{u}\right)\mathrm{d}F^{+}(u)\ll x\int_{1^{-}}^{\infty} \frac{\mathrm{d}F^{+}(u)}{u}\ll x\log^{\varepsilon}x.
$$
\end{proof}

\begin{proof}[{Proof of Proposition \ref{Mep2}.}] We may assume that $0<a\leq 1$. We will show that
\begin{equation}
\label{Meeq3.3}
\int_{1^{-}}^{x} \log u \: \mathrm{d}M^{\pm}(u)\ll x\log ^{1-b}x,
\end{equation}
whence integration by parts yields \eqref{Meeq3.1}. In fact, the bound \eqref{Meeq3.3} implies
\[
M^{\pm}(x)\ll 1+ \frac{x}{\log^{b}x} +\int_{2}^{x} \left(-\frac{1}{\log u}\right)' u\log^{1-b}u\: \mathrm{d}u \ll \frac{x}{\log^{b}x}.
\]

Our starting point is applying the operator $L$ to the relation $\mathrm{d}M^{\pm}= \mathrm{d}F^{\pm} \ast \mathrm{d}M_0$, so that we obtain (cf. \eqref{Meeq2.1})
\[ 
\int_{1^{-}}^{x} \log u \: \mathrm{d}M^{\pm}(u)=  \int_{1^{-}}^{x} ( L \mathrm{d}M_{0}) \ast \mathrm{d}F^{\pm}+ \int_{1^{-}}^{x}  \mathrm{d}M_{0} \ast \mathrm{d}H^{\pm} =: I_{1}(x)+ I_2(x). 
\]
We keep  $0<\varepsilon<  a-b$.

To estimate $I_1(x)$, we first notice that 
\[\int_{1^{-}}^{x}\log u \: \mathrm{d}M_{0}(u)\ll x\log^{1-a}x.  \]
Using Lemma \ref{Mel2}, we find
\begin{align*}
I_1(x)&\ll \int_{1^{-}}^{x} (x/u)\log^{1-a} (x/u) \mathrm{d} F^{+}(u)
\\
&
\leq x \log^{1-a} x \int_{1^{-}}^{x} \frac{ \mathrm{d} F^{+}(u)}{u}
\\
& \ll x \log^{1-b}x.
\end{align*}

We now deal with $I_2$. Applying Lemma \ref{Mel2}, the hypothesis \eqref{Meeq3}, and integration by parts, we obtain
\begin{align*}
I_2(x)&\ll \int_{1^{-}}^{x} \frac{x/u}{\log^{a} (e+x/u)} \mathrm{d} H^{+}(u)
\\
&
\ll x\log^{\varepsilon} x+ \int_{1}^{x}\left(- \frac{x/u}{\log^{a} (e+x/u)} \right)' u\log^{\varepsilon} u \: \mathrm{d}u
\\
&
\ll x\log^{\varepsilon}x +  x \int_{1}^{x} \frac{\log^{\varepsilon} u}{\log^{a} (e+x/u)}  \frac {\mathrm{d}u}{u}
\\
&=x\log^{\varepsilon}x +  x \int_{1}^{x} \frac{\log^{\varepsilon} (x/u)}{\log^{a} (e+u)}  \frac {\mathrm{d}u}{u}
\\
&\leq x\log^{\varepsilon}x +  x \log^{\varepsilon}x  \int_{1}^{x} \frac{ \mathrm{d}u}{u\log^{a} (e+u)}  \ll x \log^{1-b}x.
\end{align*}

\end{proof}

We end this section with two remarks.

\begin{remark}
\label{Mer1} Let $\ell$ be a Karamata slowly varying function \cite{bgt,korevaarbook}.
If we replace \eqref{Meeq3.0} in Lemma \ref{Mel1} by the stronger hypothesis 
\begin{equation}
\label{Meeq3.5}
\int_{1^{-}}^{\infty}\frac{|\mathrm{d}(G_1-G_2)(u)|}{u^{\sigma_{0}}}<\infty \qquad \mbox{for some } \sigma_0<1,
\end{equation}
 then $M_1(x)=O(x\ell(x))$ if and only if $M_2(x)=O(x\ell(x))$. This follows as in the proof of Lemma \ref{Mel1} if we make use of Potter's estimates \cite[p.~25]{bgt}.
\end{remark}
\begin{remark}
\label{Mer2} 
If one strengthens the condition \eqref{Meeq5} in Theorem \ref{Meth1} to
\begin{equation}
\label{Meeq3.6}
\int_{1^{-}}^{\infty}\frac{|\mathrm{d}R(u)|}{u^{\sigma_{0}}}<\infty \qquad \mbox{for some } \sigma_0<1,
\end{equation}
then the stronger bound ($b<\min\{a,1\}$)
$$
M(x)= O\left(\frac{x}{\log^{b}x}\right)
$$
 must hold true, as one infers from Remark \ref{Mer1} and Proposition \ref{Mep2}.
\end{remark}

\section{Proof of Proposition \ref{Mep1}}\label{Sec 4 Me}
Set 
$$
\mathrm{d}A(u)= \frac{\chi_{[e^{e}, \infty)}(u)}{\log u \log_{2} u}  \mathrm{d}u, \quad
\mathrm{d}B^{\pm}=\ec(\pm \mathrm{d}A), \quad \mbox{and } \quad
m_{K}(x)=\int_{1^{-}}^{x} \frac{\mathrm{d}M_{K}(u)}{u}. 
$$
We will prove Proposition \ref{Mep1} by studying the asymptotic behavior of $B^{\pm}$. The functions $m_K$ and $B^{-}$
are closely connected.

\begin{lemma}
\label{Mel4.1} We have
$$
m_{K}(x)=\frac{B^{-}(x)}{x}.
$$
\end{lemma}
\begin{proof}
It is well known (cf. \cite{diamond-zhangbook}) that
\[
\ec\left(- \frac{1-u^{-1}}{\log u}\, \mathrm{d}u \right)= \delta_1 -\frac{\mathrm{d}u}{u},
\]
where $\delta_1$ is the Dirac delta concentrated at 1. Hence,
$$\frac{1}{u}\ec\left(- \frac{1-u^{-1}}{\log u}\, \mathrm{d}u \right)= \delta_1 -\frac{\mathrm{d}u}{u^{2}}$$
has primitive $1/x$. Using this formula, one obtains
\[
\int_{1^{-}}^{x} \frac{ \mathrm{d}{M_{K}(u)}}{u}= \int_{1^{-}}^{x}\left( \delta_1 -\frac{\mathrm{d}u}{u^{2}}\right) \ast \left(\frac{ \mathrm{d}{B^{-}(u)}}{u}\right)= \int_{1^{-}}^{x} \frac{u}{x} \cdot \frac{ \mathrm{d}{B^{-}(u)}}{u}= \frac{B^{-}(x)}{x}.
\]
\end{proof}

The verification of the next lemma is left to the reader.

\begin{lemma}
\label{Mel4.2} There are constants $c_1$ and $c_2$ 
such that, as $\sigma\to 1^{+}$, 
\[
\int_{e^{e}}^{\infty} \frac{\mathrm{d}A(u)}{u^{\sigma}} = \log_2\left(\frac{1}{\sigma-1}\right) +c_1 + \frac{c_2}{\log \left(\frac{1}{\sigma-1}\right)}+ o\left(\frac{1}{\log \left(\frac{1}{\sigma-1}\right)}\right).
\]
\end{lemma}

We then have,

\begin{lemma}
\label{Mel4.3}  The bound
\begin{equation}
\label{Meeq4.1}
S(x):=\int_{1^{-}}^{x} \frac{\mathrm{d}B^{-}(u)}{u}=o(1)
\end{equation}
holds.
\end{lemma}

\begin{proof} In view of Lemma \ref{Mel4.2},
\[\int_{1^{-}}^{\infty} \frac{\mathrm{d}B^{-}(u)}{u^{\sigma}}= \exp\left(-\int_{1^{-}}^{\infty} \frac{\mathrm{d}A^{-}(u)}{u^{\sigma}}\right)= o(1)\]
and 
\[\int_{1^{-}}^{\infty} \frac{\mathrm{d}B^{+}(u)}{u^{\sigma}}= \exp\left(\int_{1^{-}}^{\infty} \frac{\mathrm{d}A^{+}(u)}{u^{\sigma}}\right)= b_1\log\left(\frac{1}{\sigma-1}\right) + b_2+ o(1)\]
for some $b_1$ and $b_2$ and where the asymptotic formulas are as $\sigma\to1^{+}$. The measure
\[\mathrm{d}B= 2\cosh^{\ast}(\mathrm{d}A):=\mathrm{d}B^{+}+\mathrm{d}B^{-}= 2\sum_{n=0}^{\infty} \frac{\mathrm{d}A^{\ast 2n}}{(2n)!}\]
is non-negative and its Mellin transform has the same real asymptotic behavior as that of the Mellin transform of $\mathrm{d}B^{+}$, that is, 
\[\int_{1^{-}}^{\infty} \frac{\mathrm{d}B(u)}{u^{\sigma}}= \int_{1^{-}}^{\infty} \frac{\mathrm{d}B^{+}(u)}{u^{\sigma}}+o(1)= b_1\log\left(\frac{1}{\sigma-1}\right) + b_2+ o(1).\]

We now apply de Haan's second order regular variation theory, which allows us to conclude (cf. \cite[Theorem~3.9.1, p.~172]{bgt} or \cite[Theorem~IV.12.1, p.~200]{korevaarbook}) that 
$$\int_{1^{-}}^{x} \frac{\mathrm{d}B^{+}(u)}{u}= b_1\log_2 x + b_2+b_1\gamma+o(1)$$
and
\[ \quad \int_{1^{-}}^{x} \frac{\mathrm{d}B(u)}{u}= b_1\log_2 x + b_2+b_1\gamma+o(1),\]
with $\gamma$ the Euler-Mascheroni constant, whence \eqref{Meeq4.1} follows.

\end{proof}

We are ready to show Proposition \ref{Mep1}.
\begin{proof}[Proof of Proposition \ref{Mep1}]

We will establish the bound
\begin{equation}\label{Meeq4.2}
B^{-}(x)=o\left(\frac{x }{\log x }\right).
\end{equation}
Proposition \ref{Mep1} is an easy consequence of \eqref{Meeq4.2}. Indeed, by Lemma \ref{Mel4.1} and \eqref{Meeq4.2}, $m_{K}(x)=o(1/\log x)$. We would therefore obtain
$$M_{K}(x)=\int_{1^{-}}^{x} u\mathrm{d}m_{K}(u) = xm_{K}(x)- \int_{1^{-}}^{x} m_{K}(u)\mathrm{d}u=o\left(\frac{x}{\log x}\right).
$$

To prove \eqref{Meeq4.2}, we slightly adapt the argument employed at the end of the proof of Lemma \ref{Mel2}. Let 
$\mathrm{d}G(u)= L\mathrm{d}A(u)$. Observe that 
\[G(x)= \frac{x}{\log_{2}x} +O\left(\frac{x}{\log x\log^{2}_{2}x}\right) .\] Using \eqref{Meeq4.1}, Chebyshev's identity (cf. \eqref{Meeq2.2}), and integration by parts,
\begin{align*}
\int_{1}^{x} \log u\: \mathrm{d}B^{-}(u)&= - \int_{1^-}^{x} \mathrm{d}G \ast \mathrm{d}B^{-}
= - x \int_{1^-}^{x} \frac{G(x/u)}{x/u} \frac{\mathrm{d}B^{-}(u)}{u}
\\
&=  \int_{1^-}^{x} \frac{x}{u} S\left(\frac{x}{u}\right)\left(\frac{{G}(u)}{u}\mathrm{d}u -\mathrm{d}G(u)\right)
\\
&
\ll \int_{e^e}^{x} o\left(\frac{x}{u}\right)  \frac{\mathrm{d}u}{\log u\log^{2}_{2}u}=o(x).
\end{align*}
Finally,
\begin{align*}
B^{-}(x)&=\frac{1}{\log x}\int_{1}^{x} \log u\: \mathrm{d}B^{-}(u) + \int_{1}^{x} \left(\int_{1}^{t} \log u\: \mathrm{d}B^{-}(u) \right) \frac{ \mathrm{d}t}{t\log^{2}t}
\\
&= 
o\left(\frac{x }{\log x }\right) + o\left(\frac{x }{\log^{2} x }\right) = o\left(\frac{x }{\log x }\right).
\end{align*}

\end{proof}

\end{document}